\title{ On a set of some recent contributions to energy equality for the Navier-Stokes equations.}
\author{Hugo Beir\~{a}o da Veiga$^{1,}$ \footnote{Partially supported  by FCT (Portugal) under the project: UIDB/MAT/04561/2020.}
\qquad Jiaqi Yang$^{2,}$\footnote{Hugo Beir\~{a}o da Veiga (\texttt{hbeiraodaveiga@gmail.com}) and Jiaqi Yang (\texttt{yjq@nwpu.edu.cn})}}
\date{
\small $^1$ Department of Mathematics, Pisa University, Pisa, Italy\\
\small $^2$ School of Mathematics and Statistics,
Northwestern Polytechnical University,
Xi'an, 710129, China}
\documentclass[12pt]{article}
\usepackage{amsfonts}
\usepackage{mathrsfs}
\usepackage{amsmath}
\usepackage{amssymb,extarrows}
\usepackage{multicol}
\usepackage{float}
\usepackage{makeidx}
\usepackage{layout}
\usepackage{array}
\usepackage{a4wide}
\usepackage{boxedminipage}
\usepackage{hyperref}
\usepackage{latexsym}
\usepackage{color}
\usepackage{dsfont}
\usepackage{graphicx}
\usepackage{ulem}
\usepackage{amsthm}

\newtheorem{theorem}{Theorem}[section]
\newtheorem{proposition}[theorem]{Proposition}

\theoremstyle{remark}
\newtheorem{remark}{Remark}[section]

\theoremstyle{definition}
\newtheorem{definition}[theorem]{Definition}

\numberwithin{equation}{section}

\newcommand{\p}{\partial}

\newcommand{\R}{\mathbb{R}}

\newcommand{\f}{\frac}
\newcommand{\n}{\nabla}
\newcommand{\tha}{\theta}

\newcommand{\ed}{\end{document}}

\newcommand{\na}{{\nabla}}
\newcommand{\pa}{{\partial}}

\newcommand{\Om}{{\Omega}}

\begin{document}
\maketitle
\begin{abstract}

In these notes we want, in addition to presenting some recent results, to both clean up and refine some reflections on a couple of articles published on paper a few years ago, 2019-20. These papers concerned integral sufficient conditions on $u\,,$ $\n u\,,$ and mixed, to guarantee the equality of the energy, EE in the sequel, for solutions of the Navier-Stokes equations under the classical non-slip boundary condition. Concerning the $\n u$ case a crucial role was enjoyed by a previous well known Berselli and Chiodaroli's pioneering 2019 work on the subject. The above three papers are the main sources of these notes. References will be mostly concentrated on their direct relation to the above papers at the time of pubblication. More recent results will be not stated throughout the article. However, in the last section, the reader will be suitably sent to the more recent bibliography.\par%
Below, we also turn back to the innovative interpretation of some main parameters which allowed to overcome their apparent incongruence.\par%
Non-Newtonian fluids were also considered in our 2019 paper, maybe for the first time in the above Berselli-Chiodaroli's particular $\n u\,$ context. However we will stick mostly to the Newtonian case since in the end we come to the conclusion that there are not particular additional obstacles to extend the present results from Newtonian to non-Newtonian fluids. Hence we avoid to go further in this direction.%
\end{abstract}


\vspace{0.2cm}

\vspace{0.2cm}

\section{Introduction. The energy equality and inequality.}
For a classical treatment of the energy equality problem for the Navier-Stokes equations we may recommend the fourth chapter of Galdi's work \cite{galdi-2000}. This reference also provides an expert treatment of the most important classic topics on the Navier-Stokes equations. Another main reference on the subject is Temam's famous treatise \cite{Tem}.\par%
In the following we are interested on integral assumptions which guarantees the energy equality (EE in the sequel) for weak (Leray-Hopf) solutions to Newtonian incompressible fluids
\begin{equation}\label{eq:NS}
\begin{cases}
\pa_t u+\,u\cdot\nabla\,u-\Delta u+\nabla \pi=0\,,\quad &\textrm{in $\Omega\times(0,T)$}\,,\\
\nabla \cdot\,u=0\,,\quad& \textrm{in $\Omega\times(0,T)$}\,,\\
u=0\,,\quad& \textrm{on $\pa\Omega\times(0,T)$}\,,\\
u(\cdot,\,0)=\,u_0\,, \quad &\textrm{in $\Omega$}\,,
\end{cases}
\end{equation}
where $\,\Om\,$ is a bounded, smooth domain in $\,\R^3\,.$ The energy equality reads
\begin{equation}\label{eq:EE}
\int_{\Omega}|u(t_0)|^2dx+2\int_0^{t_0}\int_{\Omega}|\na u(\tau)|^2\,dxd\tau=\int_{\Omega}|u_0|^2dx\,,\quad \textrm{for any} \quad t_0\in[0,T).
\end{equation}

\vspace{0.2cm}

Presently we have arrived at the conclusion that comparing the effectiveness of distinct integrability conditions yielding the energy equality is not suitable, or needs further ideas, when the number of space derivatives is different, or when the type of spaces to compare is too different (comparison between Sobolev and H\"older spaces, for example). This is now quite evident, and we have decided not to loose too much room for explanation in this regard. Since we have been interested on assumptions on the solution $u$ itself and on its first derivatives $\n u\,,$ we separate assumptions of $u-$type and assumptions of $\n u-$type, and appeal to two distinct SH numbers, $\theta$ and $\theta^{*},$ as useful tools. This useful device had not been used in reference \cite{BV-JY}.
\section{Assumptions of $\,u-$type. The Newtonian case.}
In reference \cite{shinbrot}  M.Shinbrot shows that if a weak Leray-Hopf solution $\,u\,$ to the Navier-Stokes equations \eqref{eq:NS} satisfies
\begin{equation}\label{shin}
u \in\,L^p (0,\,T; L^s(\Omega))\,,
\end{equation}
where
\begin{equation}\label{shonas}
\frac 2p + \frac 2s =\,1\,,\quad  \textrm{and} \quad s\geq\,4,
\end{equation}
then $\,u\,$ satisfies the energy equality. This result is a generalization of previous results due to G.Prodi \cite{Prodi}  and J.L.Lions \cite{lions}, where these authors proved the above result for $p=\,s=\,4\,.$\par%
For convenience we write the condition \eqref{shin}, \eqref{shonas} in the equivalent form
\begin{equation}\label{shin1}
u \in\,L^{\f{2s}{s-2}}(0,\,T; L^s(\Omega))\,,\quad s\geq 4\,.
\end{equation}
\begin{definition}
For $u$ satisfying \eqref{shin} and \eqref{shonas}, we call \emph{Shinbrot number} (SH) the quantity
\begin{equation}\label{shinum}
\tha(p,s)=:\frac 2p + \frac 2s\,.
\end{equation}
For $u$ satisfying
\begin{equation}\label{shiq}
u \in\,L^p (0,\,T; W^{1,\,q}(\Omega))\,,
\end{equation}
with $1<q<3\,,$ we define the SH number
\begin{equation}\label{shindo}
\tha^{*}(p,q)=:\frac 2p + \f{2}{q^{*}}\,, \quad q^{*}= \f{3q}{3-q}\,.
\end{equation}
\end{definition}
The role in \eqref{shindo} of Sobolev's embedding theorem $W^{1,q} \subset L^{q^{*}}$ is clear.\par%
In the following $H$ denotes the completion of $\mathcal{V}=\{\phi \in C^{\infty}_0(\Om):\,\n\cdot \phi=\,0\,\}$ in $L^2(\Omega)$.\par%
In reference \cite{BV-JY-SHINB} Theorem 1.1 we have shown that Shinbrot's criteria follows trivially, by interpolation, from the $\,L^4( (0\,,T)\times\Om))$ particular case. Moreover we have extended Shinbrot's result to the case (ii) below, as follows.
\begin{theorem}\label{theo-main}
Let $u_0 \in H $  and let $u$ be a Leray-Hopf weak solution of the Navier-Stokes equations \eqref{eq:NS}. Assume that $u$ satisfies one of the two conditions below:
\begin{equation}\label{pqcases}
\begin{cases}
(i)\quad  u \in\,L^p(0,\,T; L^s(\Omega))\,,\quad s\in [4,\,\infty]\,, \quad p=\f{2s}{s-2}.\\
(ii)\quad  u \in\, L^p(0,\,T; L^s(\Om))\,, \quad  s\in [3,\,4]\,, \quad p=\frac{s}{s-3}.
\end{cases}
\end{equation}
Then $u$ satisfies the energy equality
\begin{equation}\label{eneq}
\| u(t_0)\|_2^2 \,+\, 2\,\int_{0}^{t_0} \,\|\n u(\tau)\|_2^2 \,d\tau=\,\| u_0\|_2^2\,,
\end{equation}
for any $t_0\in [0,T)$, with $\tha=1$ in case (i) and $\,\tha= 2-\f{4}{s}<1\,$ in case (ii).
\end{theorem}
The results in item (i) and (ii) glue for the common value $r=4$. It looks significant that in correspondence to the extreme values $s=3$ and $s=\infty$, and \emph{only} for these two values, the assumptions imply the result simply by appealing to the classical Ladyzhenskaya-Prodi-Serrin (L-P-S) type conditions for regularity
\begin{equation}\label{LPS}
u \in\,L^p (0,\,T; L^q(\Omega))\,,\quad \frac 2p + \frac 3q =\,1\,,\quad \textrm{for} \quad q \in [3,\,\infty].
\end{equation}
In fact, our conditions become $u \in L^\infty(L^3)$ and $u \in L^2(L^\infty)\,,$ respectively.%

\vspace{0.2cm}

The result in item (ii) succeed in proving, as a particular case, the result stated in item (j), Theorem 3.3 below, for $\,q\leq \f{12}{7}\,$. In fact the following result holds.
\begin{proposition}\label{corshin}
Assume that
\begin{equation}\label{coros}
\na u \in\, L^\f{q}{2q-3}  (0,\,T; L^q(\Om))\,, \quad  \f32 \leq q \leq \f{12}{7}\,.
\end{equation}
Then item (ii) in Theorem \ref{theo-main}, namely,
\begin{equation}\label{esses}
u \in\, L^\f{s}{s-3}(0,\,T; L^s(\Om))\,, \quad  3\leq s \leq 4\,,
\end{equation}
holds true.
\end{proposition}
\begin{proof}
By a Sobolev's embedding theorem it follows that \eqref{coros} implies that
\begin{equation}\label{xis}
u \in\, L^\f{q}{2q-3}(0,\,T; L^\f{3q}{3-q}(\Om))\,,\quad \f32 \leq q \leq \f{12}{7}\,,
\end{equation}
which is equivalent to \eqref{esses}, as shown by setting $\,s=\f{3q}{3-q}$.
\end{proof}
Note that the full range in item (ii) was used in the proof. Moreover we have shown that \eqref{xis} yields EE.

\vspace{0.2cm}

\begin{remark}\label{lembrar}
We have just shown two partial generalizations of the above referred item (j). Question: Which of the two extensions is more general? Actually, item (k) would be much stronger than \eqref{xis} by assuming that it is equivalent to $\, u \in\, L^{\f{6q}{5q-6}} (L^{\f{3q}{3-q}})\,$. However it does not even imply the (weaker) assumption \eqref{xis}, which yields EE either. This remark has a general character, applicable each time we appeal to embedding theorems, which may lead to interesting, but rough, comparisons.
\end{remark}

\vspace{0.2cm}

Is is worth noting that for $ 3\leq s \leq 4\,$ the SH number $\tha^{*}$ of assumption \eqref{coros}
(write it in $s-$terms) is always equal to $1\,.$ However the more general assumption (ii), which yields \eqref{coros}, has a "worse" SH number $\tha=\,2-\f4s <\,1\,.$ This shows, once more, that comparison between  $\tha$ and $\tha^{*}$ by means of Sobolev's embedding Theorem is questionable.
\section{Assumptions of $\,\n u-$type in the non-Newtonian case.}
In a first overlook the reader may skip this section and go directly to the next one. Concerning the mathematical theory in the non-Newtonian case we refer, for instance, to \cite{Galdi-nonew}).\par%
Assumptions of $\,u-$type in the non-Newtonian case have been studied by J. Yang in reference \cite{Yang}. On the other hand, in reference \cite{BC}, the authors introduce sufficient conditions to obtain the EE, in terms of $\n u$. These facts led us to extend in \cite{BV-JY} the results proved in \cite{BC} to weak solutions to non-Newtonian ($r\neq2$) incompressible fluids, namely
\begin{equation}\label{eq:non-Newtonian system}
\begin{cases}
u_t+u\cdot\n u-\text{div}\left(|D(u)|^{r-2}D(u)\right)+\n \pi=0, &\text{in $\Omega\times(0,T)$},\\
\n\cdot\, u=0, &\text{in $\Omega\times(0,T)$},\\
u=0,&\text{in $\p\Omega\times(0,T)$},\\
u(\cdot,0)=u_0, &\text{in $\Omega$},
\end{cases}
\end{equation}
where
\[
D(u)=\f{1}{2}\left(\n u+(\n u)^T\right)
\]
is the symmetric part of the velocity gradient, and $\Omega\subset\R^3$ is a bounded domain, with smooth boundary $\p\Omega$.\par%
Note that we consider an extra stress tensor of the "singular" form
$$
(\mu +\,|D(u)|^{r-2}) D(u) \quad \textrm{with} \quad \mu=\,0\,.
$$
In the simpler case $\mu >\,0\,$ the proofs still hold.

\vspace{0.2cm}

The energy equality reads
\begin{equation}\label{eq:energ}
\int_{\Omega}|u(t_0)|^2dx+2\int_0^{t_0}\int_{\Omega}|D(u)(\tau)|^r\,dxd\tau=\int_{\Omega}|u_0|^2dx\,,
\end{equation}
for any $t_0\in[0,T)$.\par%

The very basic result in reference \cite{BV-JY} was its Proposition 4.1. However, after an appropriate review, we arrived to the following formulation, the Theorems 4.3 in \cite{BV-JY}.
\begin{theorem}\label{thm:largelev}
Let $u_0\in H$ and let $u$ be a Leray-Hopf weak solution of \eqref{eq:non-Newtonian system}, for some finite $\,r>\,\f85\,,$ in a smooth bounded domain $\Omega$. Let us assume that $\n u\in L^p(0,T;L^q(\Omega))$ for the following ranges of the exponents $p\,,q$:
\begin{description}
  \item[(i)$_1$] $\f{9}{5}< r \leq 2\,$, $\f{9-3r}{2}<q\leq\f95\,,$ and $\n u\in L^{\f{q(5r-9)}{3r+2q-9}}(0,T;L^q(\Omega))$;
  \item[(i)$_2$] $2<r<\f{11}{5}$, $\f{3r}{5r-6}\leq q\leq\f95\,,$ and $\n u\in L^{\f{q(5r-9)}{3r+2q-9}}(0,T;L^q(\Omega))$;
 \item[(ii)$_1$]  $\,r<\f{11}{5},\,$ $\f95<\,q\,$, and $\n u\in L^{\f{5q}{5q-6}}(0,T;L^q(\Omega))$\,;
  \item[(ii)$_2$]  $\,r\geq \f{11}{5}$.
\end{description}
Then $u$ satisfies the energy equality \eqref{eq:energ}.
\end{theorem}
Another main result in reference \cite{BV-JY} is its Theorem 4.4 (Theorem \ref{thm:BV} below), which applies also for $r\neq 2\,.$ See section \ref{chave} below.\par%
\section{Assumptions of $\,\n u-$type in the Newtonian case.}
In references \cite{BC} and \cite{BV-JY} the energy equality is studied under conditions of type \eqref{shiq} instead of type \eqref{shin1}. The starting point was reference \cite{BC} where the authors gave a very important contribution to the study of the energy equality to solutions of the Navier-Stokes equations by improving, in a quite substantial way, several previous known results. Their Theorem 1 states the following result.

\begin{theorem}\label{bers-chio}
(BERSELLI and CHIODAROLI) Let $u_0 \in H $ and let $u$ be a Leray-Hopf weak solution of the Navier-Stokes equations \eqref{eq:NS} below, and assume that $\,\na u \in\, L^p(0,\,T; L^q(\Om))\,$ for the following ranges of the exponents $p,\,q$:
\begin{equation}\label{pqcases}
\begin{cases}
(j)\quad \frac32 <q< \frac95 \,, \quad \textrm{and}\quad \na u \in\, L^{\frac{q}{2q-\,3}}(0,\,T; L^q(\Om))\,\,;\\
(jj)\quad \frac95 \leq q<3 \,, \quad \textrm{and}\quad \na u \in\, L^{\frac{5q}{5q-\,6}}(0,\,T; L^q(\Om))\,\,;\\
(jjj)\quad 3\leq q \,, \quad \textrm{and}\quad \na u \in\, L^{1+\,\frac{2}{q}}(0,\,T; L^q(\Om))\,\,.\\
\end{cases}
\end{equation}
Then $u$ satisfies the energy equality
\begin{equation}\label{eneq}
\| u(t_0)\|_2^2 \,+\, 2\,\int_{0}^{t_0} \,\|\n u(\tau)\|_2^2 \,d\tau=\,\| u_0\|_2^2\,,
\end{equation}
for any $t_0\in [0,T)$.
\end{theorem}
By restriction of the above Theorem \ref{thm:largelev} (Theorem 4.3 in reference \cite{BV-JY}) to the case $r=2$ we obtain the following result:
\begin{proposition}\label{thm:largelev-N}
Let $u_0\in H$ and let $u$ be a Leray-Hopf weak solution of \eqref{eq:NS} in a smooth bounded domain $\Omega$. Let us assume that $\n u\in L^p(0,T;L^q(\Omega))$ for the following ranges of the exponents $p\,,q$:
\begin{description}
  \item[(i)$_1$] $\f{3}{2}<q\leq\f95\,,$ and $\n u\in L^{\f{q}{2q-3}}(0,T;L^q(\Omega))$;
 \item[(ii)$_1$] $\f95<\,q\,$, and $\n u\in L^{\f{5q}{5q-6}}(0,T;L^q(\Omega))$\,;
\end{description}
Then $u$ satisfies the energy equality \eqref{eq:energ}.
\end{proposition}
\begin{remark}\label{rem:G-N}
For $q>3\,$ the item(ii)$_1$ improves the item (jjj) in Theorem \ref{bers-chio}. Furthermore the item (j) in this last theorem is improved by item (k) below, for $\,q\leq \f{12}{7}\,,$ by appealing to the result stated in the Proposition \ref{corshin}. This allow us to replace this proposition by the following stronger result.
\end{remark}
\par%
\begin{theorem}\label{thm:largelev-New}
Let $u_0\in H$ and let $u$ be a Leray-Hopf weak solution of \eqref{eq:NS} in a smooth bounded domain $\Omega$. Let us assume that $\n u\in L^p(0,T;L^q(\Omega))$ for the following ranges of the exponents $p\,,q$:
\begin{description}
\item[(k)] $\f32 \leq q \leq \f{12}{7}\,,$ and
$\n u \in\, L^\f{6q}{5q-6}(0,\,T; L^q(\Om))\,;$
 \item[(kk)] $\f{12}{7}<q\leq\f95\,,$ and $\n u\in L^{\f{q}{2q-3}}(0,T;L^q(\Omega))$;
 \item[(kkk)] $\f95<\,q\,$,  and  $\n u\in L^{\f{5q}{5q-6}}(0,T;L^q(\Omega))$\,;
\end{description}
Then $u$ satisfies the energy equality \eqref{eq:energ}.
\end{theorem}
\begin{remark}\label{rem:B-GE}
In the theorem 2 in reference \cite{BE-GEO} it is essentially shown, for the torus and $q>2\,,$ that results of the above kind also apply to the Euler equations. A quite unexpected result.
\end{remark}
\begin{subsection}{On a very general result obtained by interpolation.}
We start by quoting the unpublished 2008 paper "On the energy equality for the Navier-Stokes problem", by Carlo R. Grisanti, where this author proved quite interesting results in the wake of reference \cite{CFS}. Note that $m$ is a positive real, not just an integer. The main Grisanti's result was the following.
\begin{theorem}\label{th:grisanti}
(C.R. GRISANTI)  Let $u$ be a weak solution of the Navier-Stokes equations \eqref{eq:NS}. Then the energy equality holds in the following cases:
\begin{equation}\label{carlo}
\begin{cases}
(j)\quad u \in\, L^{\frac{5+ 12m}{6m}}(0,\,T; W^{m,2}(\Om))\,, \quad \f56 \leq m<\,1\,;\\
(jj)\quad u \in\, L^{\frac{5}{2m}}(0,\,T; W^{m,2}(\Om))\,, \quad  1\leq m<\,\f54\,;\\
(jjj)\quad u \in\, L^2(0,\,T; W^{m,2}(\Om))\,, \quad \f54 \leq m<\,\f32\,.
\end{cases}
\end{equation}
\end{theorem}
However the referee remarked that the same results could be obtained by interpolation. In fact, in the proposition below, obtained by this last technique, Grisanti's unpublished results are even improved. The following result is the Proposition 4.6 in reference \cite{BV-JY}.
\begin{theorem}\label{pron}
Let $u$ be a Leray-Hopf weak solution of the Navier-Stokes equations \eqref{eq:NS}, and assume that one of the following assumptions hold.
\begin{equation}\label{pqcases}
\begin{cases}
(i)\quad u \in\, L^{\f{2}{2m-1}}(0,\,T; W^{m,2}(\Om))\,,\quad \textrm{for} \quad  \f12 <m \leq \f56;\\
(ii)\quad u \in\, L^{\f{5}{2m}}(0,\,T; W^{m,2}(\Om))\,,\quad \textrm{for} \quad  m\geq \f56.\\
\end{cases}
\end{equation}
Then $\,u\,$ satisfies the energy equality.
\end{theorem}
\end{subsection}
%
\begin{section}{A new reading key of Shinbrot's type numbers.}\label{chave}
This chapter is of particular interest here. It concerns the comparison of the strength of distinct results by appealing to the SH-numbers. The classical SH number $\tha\,$ looks significant in the $u-$case, but in $\n u-$case the SH-number $\tha^{*}\,$ is totally inappropriate. In fact the SH number $\tha^{*}\,$ obtained simply from $\tha\,$ by appealing to Sobolev's embedding, shows a dramatic dependence on $q$. However a new point of view gives meaning to this attempt, as shows in the Theorem below. As still claimed in reference \cite{BV-JY} this result holds for $r\neq 2\,$, with essentially the same proof (which does not use the $r-$weak estimate).
\begin{theorem}\label{thm:BV}
Let $u_0\in H\,,$ and let $u$ be a Leray-Hopf weak solution of \eqref{eq:NS} in a smooth bounded domain $\Omega$. Further assume that one of the two following assumptions hold.\par%
\begin{description}
  \item[(i)] $u\in L^{p_1}(0,T;L^{q_1}(\Omega))\cap L^{\f{9q}{8q-9}}(0,T;W^{1,q}(\Omega))\,,\quad $\\ with $\quad \frac{2}{p_1}+\,\frac{2}{q_1}=\,\frac{10}{9}\,,\quad$
  $\f98<q\,,\quad $ and $\quad q_1 \geq\,\f{2 q}{q-\,1}\,.$
  \item[(ii)] $ u\in L^{\f{5q}{5q-6}}(0,T;W^{1,q}(\Omega))\,,\quad$ with $\quad \f95<\,q\,.$
\end{description}
Then $u$ satisfies the energy equality
\begin{equation*}
\|u(t_0)\|_2^2+2\int_0^{t_0}\|D(u)(\tau)\|^2_2 d\tau=\|u_0\|_2^2\,,
\end{equation*}
for all  $t_0 \in\,[0,T).$
\end{theorem}
\begin{remark}
In item (i) both spaces enjoy the same Shinbrot's number $\,\tha=\tha^{*}=\frac{10}{9}\,,$
for all $q < 3\,.$ Hence our main aim is fulfilled: Independently of the value of $q\,$ all the assumptions (i) have the same SH number $\f{10}{9}\,.$ However, for  $q>\,\f95\,,$ item (i) holds but is overshadowed by the more general item (ii). This overlap does not invalidate its broad validity (note that item (ii) has SH-number $\tha^{*}> \f{10}{9}\,$).
\end{remark}
\begin{remark}
--- For $q<\,\f95\,,$ since $\f{9q}{8q-9} < \f{q}{2q-3}\,,$ we have weakens the condition $\,\n u \in\, L^{\frac{q}{2q-\,3}}(0,\,T; L^q(\Om))\,$ up to the desired level $\theta=\,\f{10}{9}$ enjoyed by $\,\n u\in  L^{\f{9q}{8q-9}}(0,T;L^q(\Omega))\,,$ compensated by the addition of an assumption on $\,u$ which enjoys just the same level $\theta=\,\f{10}{9}$.\par%
--- For $q=\,\f95\,$ the intersection of the two spaces in item (i) coincides with the second space, and the sufficient condition becomes simply $\,u \in L^3 (0,\,T; W^{1,\,\frac{9}{5}}(\Omega))\,.$ This is just the only case that, in reference \cite{BC}, yields the strongest result (by the $\tha^{*}$ criterium), namely $\tha^{*}=\,\f{10}{9}\,,$ always reached up by the new criterium in (i) Theorem \ref{thm:BV}.\par%
--- For $q>\f95\,$ the assumption in item (i) is superseded by the broader assumption in item (ii)$_1$ of Theorem \ref{thm:largelev} since
\[
L^{p_1}(0,T;L^{q_1}(\Omega))\cap L^{\f{9q}{8q-9}}(0,T;W^{1,q}(\Omega))\subset L^{\f{9q}{8q-9}}(0,T;W^{1,q}(\Omega))\subset L^{\f{5q}{5q-6}}(0,T;W^{1,q}(\Omega))\,.
\]
\end{remark}
For this reason it seemed appropriate to have inserted this assumption in the above theorem as item (ii).

\vspace{0.2cm}

In reference \cite{BC} the authors warn against the risk of too simplified comparison between sufficient conditions fo EE in terms of distinct functional spaces, in particular between Sobolev and H\"older spaces. We strongly agree this warn. Actually, even comparison between conditions in terms of $u$ and $\n u$ via Sobolev's embedding theorems could be misleading, as we still have seen.

\vspace{0.2cm}

We end this section by quoting reference \cite{WMH} where a very general set of quite interesting results is stated. We strongly refer the reader to the original paper. However let's remark that in the Theorem 1.1 in the above reference the two functional spaces appearing in its equation (1.9) have the same Shinbrot numbers, $\theta= \tha^{*}\,$ if and only if $k=\f{9l}{8l-9}$, which is just the assumption in Theorem 4.4, item (i) in reference \cite{BV-JY}  (Theorem \ref{thm:BV} in the present notes), necessary to obtain the SH-number's interpretation described in section \ref{chave}, which was our specific purpose.%
\end{section}
\begin{section}{Some references.}
In the bounded domain case, Cheskidov, Friedlander, and Shvydkoy \cite{CFS} proved that if $\, A^{\f5{12}}u\in L^3(0,T;L^2(\Omega)) $
then $u$ satisfies energy equality. This condition looks effectively equivalent to $\, u\in L^3(0,T;W^{\f56,2}(\Omega))\,,$
which is just assumption required in the Theorem \ref{carlo}, for $m=\f56\,.$\par%
On the other hand Farwig in reference \cite{Farwig} proved EE if $\, A^{\f14}u\in L^3(0,T;L^{\f{18}{7}}(\Omega))\,,$ which is formally equivalent to $\, u\in L^3(0,T;W^{\f12,\f{18}{7}}(\Omega))\,.$\par%
Other related main references are Cheskidov and Luo \cite{CL}, Shvydkoy \cite{shvy}, and Cheskidov, Constantin, Friedlander, and Shvydkoy \cite{CCFS}. We refer to \cite{BC}, section 2.2.1 for a report on  the above results, and interesting related considerations.%

\vspace{0.2cm}

In \cite{BC} Theorem 2 energy equality is studied for distributional solutions. As the authors remarked the techniques employed were inspired by G.P.Galdi's  references \cite{galdi-dist1} and \cite{galdi-dist2} on the same subject.\par%

\vspace{0.2cm}

Let's move on to more recent results.
 For a very recent deep study on EE results for Euler and Navier-Stokes equations in different functional spaces, and on Onsager's conjecture, we refer to Berselli's contributions \cite{BE-22}, \cite{BE-GEO}, and \cite{BE-23}, the second one coauthored by Georgiadis. This set of results is important by itself but also due to the stimulating considerations done by the author. In references \cite{BE-22} and \cite{BE-23} spaces of H\"older continuous functions play a central role. As claimed by the author in the introduction, the proofs in reference \cite{BE-23} are quite accessible to a wide audience.

Recently, Yang \cite{Yang-2} studied the energy equality of axisymmetric Navier-Stokes equations. A interesting fact is that it is enough to impose the Shinbrot condition to $\tilde{u}=u^re_r+u^z e_z$. It seems that this is a first result on the energy conservation law for the axisymmetric Navier-Stokes equations.
\vspace{0.2cm}

To end we would remark that a great part of the proofs can be easily extended to many other situations. Instead, a difficult problem of great interest is, in our opinion, to find relations between energy equality and uniqueness.
\end{section}

\end{document}